\documentclass[12pt]{amsart}
\usepackage{amsmath,amssymb}
\numberwithin{equation}{section}
 \newtheorem{defn}{Definition}
 \newtheorem{prop}[defn]{Proposition}
 \newtheorem{lem}[defn]{Lemma}
 
 \newtheorem{thm}[defn]{Theorem}
\theoremstyle{remark}
\newtheorem{rem}{\bf Remark}

\newcommand{\R}{\mathbb{R}}
\newcommand{\N}{\mathbb{N}}

\newcommand{\C}{\mathbb{C}}
\newcommand{\e}{\varepsilon}

\newcommand{\lec}{\lesssim}

\newcommand{\embed}{\hookrightarrow}

\newcommand{\LR}[1]{{\langle #1 \rangle}}

\title{Blowup solutions for the nonlinear Schr\"odinger 
equation with complex coefficient}
\author{Shota Kawakami and Shuji Machihara}
\address{Department of Mathematics, Faculty of Science, Saitama University, 255 Shimo-Okubo, Sakura-ku, Saitama City 338-8570, Japan}
\email{s.kawakami.090@ms.saitama-u.ac.jp, \ machihara@rimath.saitama-u.ac.jp}
\thanks{The second author was supported by JSPS Grant-in-Aid for Scientific Research C 
[grant number 16K05191].}

\begin{document}
\maketitle
\begin{abstract}
We construct a finite time blow up solution for the nonlinear Schr\"odinger equation 
with the power nonlinearity whose coefficient is complex number.
We generalize the range of both the power and the complex coefficient  
for the result of 
Cazenave, Martel and Zhao \cite{CMZ}. 
As a bonus, we may consider the space dimension $5$. 
We show a sequence of solutions closes to the blow up profile which is a 
blow up solution of ODE. We apply the Aubin-Lions lemma for the 
compactness argument for its convergence. 
\end{abstract}

\section{nonlinear Schr\"odinger equation with complex coefficient}
We consider the following nonlinear Schr\"odinger equation with complex coefficient 
of the power nonlinearity
\begin{align}\label{nls}
iu_t(t,x)+\Delta u(t,x)=\lambda|u(t,x)|^{\alpha}u(t,x), \quad t\in\R, \ x\in\R^N,
\end{align}
where 
$i=\sqrt{-1}, u_t=\partial_tu, \Delta u=\sum_{j=1}^{N}\partial_{x_j}^2u, \alpha>0, 
\lambda\in\C\backslash\{0\}$ and $u=u(t,x):\R\times\R^N\to\C$ is a solution. 
There are large number of papers for the case $\lambda\in\R$ which delt with, for examples, 
wellposedness and behaviours of solutions. 
In this case, we have the following conservation laws, charge and energy respectively
\begin{align*}
\frac{d}{dt}\|u(t)\|_{L^2}&=0, \\
\frac{d}{dt}\left(\frac12\|\nabla u(t)\|_{L^2}^2+\frac{\lambda}{\alpha+2}\|u(t)\|_{L^{\alpha+2}}^{\alpha+2}\right)
&=0.
\end{align*}
These laws do not hold with $\lambda\in\C\backslash\R$ in general. 
In the special case $\lambda=i$, there are results in the book written by Lions \cite{quelques}, 
the technique of monotone operators and compactness argument are
applied to have existence of the solutions. 
Cazenave, Martel and Zhao \cite{CMZ} investigate \eqref{nls} with the same 
setting $\lambda=i$. 
More general setting $\lambda\in\C$ are discussed in \cite{KOS} and  
which are sometimes called complex Ginzburg-Landau equation. 
We consider \eqref{nls} under this general setting $\lambda\in\C$ with some 
assumptions. We investigate the finite time blow-up phenomena of the 
solution of \eqref{nls}. 
There are former results. Kita \cite{Kita} proved the blow-up solution which 
starts with small initial data in one spatial dimension, so called, small data blow-up 
phenomena. 
Cazenave, Martel and Zhao \cite{CMZ} proved the blow-up solution in general 
dimensions. We introduce more details in \cite{CMZ} at Remark \ref{rem1} below.

If we set the initial data $u(0,x)=f(x)$ belonging to 
the Sobolev space of order 1, that is $H^1(\R^N)$, 
from the standard argument, there exists an unique time local solution of \eqref{nls}. 
There we assume the power condition
\begin{align}\label{LWP-p}
0<\alpha\le\frac{4}{N-2} 
\end{align}
where it means actually $0<\alpha<\infty$ for $N=1,2$ and $0<\alpha\le\frac{4}{N-2}$ 
for $n\ge3$, and we employ this rule throughout this paper. 
We introduce a blow-up profile for our argument:
\begin{align}\label{b-profile}
U(t,x)=(-\alpha\text{Im} \lambda(t-|x|^k))^{\frac{i\lambda}{\alpha\text{Im}\lambda}}. 
\end{align}
This is the same in \cite{CMZ} when $\lambda=i$. From the elemental calculation, 
we have
\begin{align*}
\lim_{t\to0-0}\|U(t)\|_{H^1}=\infty. 
\end{align*}
Now we state the main theorem.
\begin{thm}\label{thm1}
Let $N=1,2,3,4,5$. Let $\alpha$ and $\lambda$ satisfy
\begin{align}\label{p-assume1}
1<\alpha\le\frac{4}{N-2}, \qquad (\alpha+2){\rm{Im}}(\lambda)\ge\alpha|\lambda|,
\end{align}
and
\begin{align}\label{p-assume2}
 (\alpha+2){\rm{Im}}(\lambda)>\alpha|\lambda| \qquad {\rm{when}}\quad
\alpha>\frac{N}{N-2}\quad{\rm{or}}\quad\alpha=\frac{4}{N-2}.
\end{align}
Then there exists a solution $u\in C((-\infty,0),H^1(\R^N))$ of \eqref{nls} which 
blows up at $t=0$ in the sense of the following
\begin{align*}
\lim_{t\to0-0}\|u(t)\|_{H^1}=\infty. 
\end{align*}
More precisely there exist positive constants $C,\delta$ and $\mu$ such that 
\begin{align*}
\|u(t)-U(t)\|_{H^1}\le C(-t)^{\mu}, \quad-\delta<t<0
\end{align*}
where $U$ is the blow-up profile of \eqref{b-profile}. 
\end{thm}

\begin{rem}
The conditions \eqref{p-assume1} and \eqref{p-assume2} allow the followings
\begin{align*}
1<\alpha<\infty,\quad &(\alpha+2){\rm{Im}}(\lambda)\ge\alpha|\lambda| \quad
\text{for} \ N=1,2. \\
1<\alpha\le3,\quad &(\alpha+2){\rm{Im}}(\lambda)\ge\alpha|\lambda| \quad 
\text{or} \\
3<\alpha\le4,\quad &(\alpha+2){\rm{Im}}(\lambda)>\alpha|\lambda| \quad
\text{for} \ N=3. \\
1<\alpha<\frac{4}{N-2},\quad &(\alpha+2){\rm{Im}}(\lambda)\ge\alpha|\lambda| \quad 
\text{or} \\ 
\alpha=\frac{4}{N-2},\quad &(\alpha+2){\rm{Im}}(\lambda)>\alpha|\lambda| \quad
\text{for} \ N=4,5. 
\end{align*}
\end{rem}

\begin{rem}\label{rem1}
Cazenave-Martel-Zhao \cite{CMZ} gave the same conclusion under the assumption 
$N=1,2,3,4$ and for the power
\begin{align*}
2\le\alpha\le\frac{4}{N-2}
\end{align*}
and the coefficinet $\lambda=i$ which satisfy \eqref{p-assume1} and \eqref{p-assume2}. 
They, in fact, proved more generalized case that any number and anywhere for the blow up 
points.  We generalize the range of $\lambda$ and we reduce the lower bound of $\alpha$. 
For our lower bound $\alpha>1$, it seems difficult to reduce the number below 1 since 
our argument requires to estimate $H^1$ norm for the difference of the two functions, 
that are, solution of \eqref{nls} and the blow up profile $U$. 
We remark $\alpha=1$ is critical and still open as well. 
\end{rem}

In the sequential paper, we will deal with the double critical point 
\begin{align*}
\alpha=\frac{4}{N-2}, \qquad (\alpha+2){\rm{Im}}(\lambda)=\alpha|\lambda|
\end{align*}
for both the time global wellposendness and the blow up problem. 
We will apply the results in \cite{KOS} to 
solve the following complex Ginzburg-Landau equation for global existence time
\begin{align}\label{cgl}
iu_t(t,x)+(1-i\e)\Delta u(t,x)=\lambda|u(t,x)|^{\alpha}u(t,x), \quad t\in\R, \ x\in\R^N.
\end{align}
The solution $u_{\e}$ exists globally in the negative time 
$u_{\e}\in C((-\infty, 0]:H^1)$ for any $\e>0$ and we take limit $\e\to0+0$.



\section{Preliminaries}
Before going to our proof, we collect the standard estimates.

\begin{lem}\label{s-est-Lem}
Let $p>0$ and $n\in\N\cup\{0\}$. Then the estimates 
\begin{align}
||z|^{p-n}z^n-|w|^{p-n}w^n|\lec 
\begin{cases}
(|z|^{p-1}+|w|^{p-1})|z-w| \quad &{\rm{if}} \ p\ge1, \\
|z-w|^{p} \qquad\qquad &{\rm{if}} \ 0<p\le1, \label{s-est2}
\end{cases}
\end{align}
hold for $z,w\in\C$ where the implicit constant depends on $p,n$ and is independent of $z,w$.
\end{lem}
We remark for this lemma that we may consider minus power $p-n<0$ of the modulus, 
although the total power $(p-n)+n=p$ is positive. 

For the nonlinear term, we set 
\begin{align*}
g_{\alpha}(u):=|u|^{\alpha}u.
\end{align*}
\begin{lem}\label{con-abst}
Let $N=1,2,\ldots$.
Let $I$ be bounded interval. Suppose a sequence $(u_n)\subset L^{\infty}(I:H^1(\R^N))$ 
and a function $u\in L^{\infty}(I:L^2(\R^N))$ satisfy 
\begin{align}
&\sup_{n\in\N}\|u_n\|_{L^{\infty}(I:H^1)}<\infty, \label{bdd1} \\
&\lim_{n\to\infty}\|u_n-u\|_{L^{\infty}(I:L^2)}=0. \label{L2c}
\end{align}
Let $\alpha$ satisfies subcritical or critical condition
\begin{align*}
0<\alpha\le\frac{4}{N-2}.
\end{align*}
Then the limit and $g_{\alpha}$ of it belong to the spaces
\begin{align*}
u\in L^{\infty}(I:H^1(\R^N)) \quad{\rm{and}}\quad 
g_{\alpha}(u)\in L^{\infty}(I:L^{\frac{\alpha+2}{\alpha+1}}(\R^N))
\end{align*}
respectively, and the following convergences hold
\begin{align}
u_n\rightharpoonup u\qquad &*\text{weak in}\quad L^{\infty}(I:H^1(\R^N)), \label{con1} \\
\Delta u_n\rightharpoonup\Delta u\qquad &*\text{weak in}\quad 
L^{\infty}(I:H^{-1}(\R^N)), \label{con2} \\
g_{\alpha}(u_n)\rightharpoonup g_{\alpha}(u)\qquad &*\text{weak in}\quad 
L^{\infty}(I:L^{\frac{\alpha+2}{\alpha+1}}(\R^N)). \label{con3}
\end{align}
Moreover, suppose additional bounded condition
\begin{align}\label{bdd2}
&\sup_{n\in\N}\|u_n\|_{W^{1,\infty}(I:H^{-1})}<\infty,
\end{align}
then the limit also belong to the space
\begin{align*}
u\in W^{1,\infty}(I:H^{-1}(\R^N)),
\end{align*}
and the following convergence holds
\begin{align}\label{con4}
\partial_tu_n\rightharpoonup\partial_tu\qquad *\text{weak in}\quad L^{\infty}(I:H^{-1}(\R^N)).
\end{align}
\end{lem}
\begin{rem}
We do not need to take any subsequence $(u_{n_k})$ in the conclusions. 
We do not require that $u_n$ satisfy any equation likely as \eqref{nls}, neither. 
\end{rem}

\begin{proof}
From \eqref{bdd1}, we have a subsequence $(u_{n_k})$ and $v\in L^{\infty}(I:H^1)$ such as
\begin{align*}
u_n\rightharpoonup v\qquad *\text{weak in}\quad L^{\infty}(I:H^1).
\end{align*}
Since limit is unique, we obtain $u=v\in L^{\infty}(I:H^1)$. If there is 
subsequence $(u_{n_k})$ which does not converge to $u$, then there are 
its subsequence  $(u_{n_{k_j}})$, 
some test function $\phi\in L^1(I:H^{-1})$ and $\delta>0$ satisfy
\begin{align}\label{contradict1}
|\LR{u_{n_{k_j}}-u,\phi}|>\delta\qquad\text{for any} \ j=1,2,\ldots.
\end{align}
This is a contradiction since this subsequence  $(u_{n_{k_j}})$ is bounded, and so, 
it contains a subsequence  $(u_{n_{k_{j_l}}})$ which does not satisfy \eqref{contradict1}. 
Therefore the whole sequence converges to the limit. We obtain \eqref{con1}, and so,  
\eqref{con2} follows as well. 
Next we consider \eqref{con3}. Incidentally, 
we have the following convergence in the norm for the subcritical power 
$(N-2)\alpha<4$. By using the Gagliard-Nirenberg inequality
\begin{multline*}
\|g(u_n)-g(u)\|_{L^{\infty}L^{\frac{\alpha+2}{\alpha+1}}}
\lec(\|u_n\|_{L^{\infty}L^{\alpha+2}}^{\alpha}+\|u\|_{L^{\infty}L^{\alpha+2}}^{\alpha})
\|u_n-u\|_{L^{\infty}L^{\alpha+2}} \\
\lec(\|u_n\|_{L^{\infty}H^1}^{\alpha}+\|u\|_{L^{\infty}H^1}^{\alpha})
\|u_n-u\|_{L^{\infty}L^{2}}^{1-\frac{N\alpha}{2(\alpha+2)}}
\|\nabla u_n-\nabla u\|_{L^{\infty}L^2}^{\frac{N\alpha}{2(\alpha+2)}}\to0
\end{multline*}
as $n\to\infty$, where $\frac{N\alpha}{2(\alpha+2)}<1$. 
For the critical power $(N-2)\alpha=4$, we shall prove the same convergence 
but in the $*$ weak sense. We take any  $\phi\in L^1(I:C^{\infty}_0(\R^N))$. 
For each $s\in I$, we calculate 
\begin{multline*}
\LR{g(u_n(s))-g(u(s)),\phi(s)}
\le\|g(u_n(s))-g(u(s))\|_{L^{\frac{(\alpha+2)(\alpha+4)}{\alpha^2+6\alpha+4}}}
\|\phi(s)\|_{L^{\frac{(\alpha+2)(\alpha+4)}{4}}} \\
\lec(\|u_n(s)\|_{L^{\alpha+2}}^{\alpha}+\|u(s)\|_{L^{\alpha+2}}^{\alpha})
\|u_n(s)-u(s)\|_{L^{\frac{\alpha+4}{2}}}\|\phi(s)\|_{L^{\frac{(\alpha+2)(\alpha+4)}{4}}} \\
\lec(\|u_n(s)\|_{H^1}^{\alpha}+\|u(s)\|_{H^1}^{\alpha})
\|u_n(s)-u(s)\|_{L^{2}}^{1-\frac{N\alpha}{2(\alpha+4)}}
\|\nabla u_n(s)-\nabla u(s)\|_{L^2}^{\frac{N\alpha}{2(\alpha+4)}}
\|\phi(s)\|_{L^{\frac{(\alpha+2)(\alpha+4)}{4}}} \\
\to0 \qquad\text{as}\quad n\to\infty.
\end{multline*}
The integrable majorant in $s$ is the following
\begin{multline*}
|\LR{g(u_n(s))-g(u(s)),\phi(s)}|
\le\|g(u_n(s))-g(u(s))\|_{L^{\frac{\alpha+2}{\alpha+1}}}
\|\phi(s)\|_{L^{\alpha+2}} \\
\lec(\|u_n(s)\|_{L^{\alpha+2}}^{\alpha+1}+\|u(s)\|_{L^{\alpha+2}}^{\alpha+1})
\|\phi(s)\|_{L^{\alpha+2}}\lec\|\phi(s)\|_{L^{\alpha+2}} \in L^1(I).
\end{multline*}
We then apply Lebesgue's dominated convergence theorem to obtain
\begin{align*}
\lim_{n\to\infty}\int_I\LR{g(u_n(s))-g(u(s)),\phi(s)}ds
=\int_I\lim_{n\to\infty}\LR{g(u_n(s))-g(u(s)),\phi(s)}ds=0.
\end{align*}
Since $L^1(I:C_0^{\infty})\embed L^1(I:L^{\alpha+2})$ is dense, this implies  
as desired
\begin{align}\label{weakstar1}
g(u_n)\rightharpoonup g(u) \quad\text{weak} * \ \text{in} \ 
L^{\infty}(I:L^{\frac{\alpha+2}{\alpha+1}}). 
\end{align}
Actually we can prove this weak $*$ convergence for $\{g(u_n)\}$ 
from $L^{\infty}H^1$ boundedness 
and $L^{\infty}L^2$ convergence of $\{u_n\}$ in another way. 
Since the interval is finite, we have
\begin{align*}
u_n\to u\in L^{\infty}(I;L^2)\embed L^2(I\times\R^N),
\end{align*}
we then have a subsequence $\{u_{n_k}\}$ which converges almost everywhere, 
i.e. 
\begin{align*}
u_{n_k}\to u \quad\text{a.e. \ in}\quad I\times\R^N.
\end{align*}
This gives
\begin{align*}
g(u_{n_k})\to g(u) \quad\text{a.e. \ in}\quad I\times\R^N.
\end{align*}
From Fubini's theorem, we have
\begin{align*}
g(u_{n_k})(s)\to g(u)(s) \quad\text{a.e.}\quad x\in\R^N
\end{align*}
for a.e. $s\in I$.
We estimate the norm
\begin{align*}
\|g(u_{n_k})(s)\|_{L^{\frac{\alpha+2}{\alpha+1}}_{x}}
\le\|u_{n_k}(s)\|_{L^{\alpha+2}_{x}}^{\alpha+1}
\le\|u_{n_k}(s)\|_{H^1_{x}}^{\alpha+1}
\end{align*}
and this is uniformly bounded in $k$ and $s\in I$. 
From Lemma 1.3 in \cite{quelques} , we have
\begin{align*}
g(u_{n_k})(s)\rightharpoonup g(u)(s) \quad\text{weak} \ \text{in} \ 
L^{\frac{\alpha+2}{\alpha+1}}(\R^N)
\end{align*}
for a.e. $s$ in $I$. Here we do not need to take a subsequence out. 
We estimate
\begin{align*}
\|g(u)(s)\|_{L^{\frac{\alpha+2}{\alpha+1}}_{x}}
\le\liminf_{k\to\infty}\|g(u_{n_k})(s)\|_{L^{\frac{\alpha+2}{\alpha+1}}_{x}}\le C
\end{align*}
with some constant $C$. So, 
for any $\phi\in L^1(I;L^{\alpha+2})$ and this $C$, we have
\begin{align*}
|
\LR{g(u_{n_k})(s)-g(u(s)),\phi(s)}
|
\lec C\|\phi(s)\|_{L^{\alpha+2}} \in L^1(I).
\end{align*}
From Lebesgue's dominated convergence theorem again gives the result, 
\begin{align*}
g(u_{n_k})\rightharpoonup g(u) \quad\text{weak} * \ \text{in} \ 
L^{\infty}(t,0;L^{\frac{\alpha+2}{\alpha+1}}).
\end{align*}
We may say from this argument that the whole sequence $(u_n)$ converges
to the same limit $u$ which corresponds to \eqref{weakstar1}. 
This complete the second proof for weak $*$ convergence for $(g(u_n))$. 

Next we assume \eqref{bdd2} additionally and show \eqref{con4}. From the same argument above,  
we have $\partial_tu\in L^{\infty}(I:H^{-1})$ where $u$ is the limit in \eqref{L2c}. 
Form \eqref{con1}, we have for any $\phi\in C^1_0(I:H^1)$,
\begin{align*}
\int_I\LR{\partial_tu_n(s)-\partial_tu(s),\phi(s)}_{H^{-1},H^1}ds
=-\int_I\LR{u_n(s)-u(s),\partial_t\phi(s)}_{H^{-1},H^1}ds\to0
\end{align*}
as $n\to\infty$. This implies 
\begin{align}\label{weakstar4}
\partial_tu_n\rightharpoonup\partial_tu \quad\text{weak} * \ \text{in} \ L^{\infty}(I:H^{-1}).
\end{align}
Combining with \eqref{con1}, we obtain \eqref{con4} and this complete the proof.
\end{proof}

We define the space
\begin{align*}
\Sigma=\Sigma(\R^N):&=\{f\in H^1(\R^N):\|f\|_{\Sigma}<\infty\}, \\
&\|f\|_{\Sigma}^2=\|f\|_{H^1}^2+\||\cdot|f\|_{L^2}^2.
\end{align*}
\begin{lem}
Let $N\in\N$. Then the embedding $\Sigma(\R^N)\embed L^2(\R^N)$ is compact.
\end{lem}
We introduce the Aubin-Lions lemma, see Simon \cite{ALS}.
\begin{lem}
Let $X_0, X$ and $X_1$ be three Banach spaces with $X_0\embed X\embed X_1$ where 
$X_0$ is compactly embedded in $X$, and $X$ is continuously embedded in $X_1$. 
For $1\le p,q\le\infty$, define
\begin{align*}
\Phi:=
\begin{cases}
\{u\in L^p(0,T:X_0) \ | \ u_t\in L^q(0,T:X_1)\}, \quad\text{if} \ p<\infty,\\
\{u\in C(0,T:X_0) \ | \ u_t\in L^q(0,T:X_1)\}, \quad\text{if} \ p=\infty.
\end{cases}
\end{align*}
Then
\begin{enumerate}
\item If $p<\infty$, the embedding $\Phi\embed L^p(0,T:X)$ is compact.
\item If $p=\infty, q>1$, the embedding $\Phi\embed C(0,T:X)$ is compact.
\end{enumerate}
\end{lem}

\section{Time global well-posedness}\label{GWP}
In this section we show the existence of solution from any initial data in $H^1$. 
We later consider the sequence of solutions $v_n, n=1,2,\ldots$ on each time interval 
$[0,T^*_n)$ where $T^*_n$ is maximal existence time. If we obtain the time global 
well-posedness, we have uniform existence time $T^*_n=\infty, n=1,2,\ldots$. 
\begin{thm}\label{WP-thm}
Let $n\in\N, \lambda\in\C$ and
\begin{align}\label{power-condi}
0<\alpha\le\frac{4}{N-2}. 
\end{align} 
Then for any $f\in H^1(\R^N)$ 
 there exists $T^*, -T_*>0$ and unique solution
\begin{align*}
u\in C((T_*,T^*),H^1(\R^N)).
\end{align*}
Moreover if we additionally assume $(\alpha+2){\rm Im}(\lambda)>\alpha|\lambda|$,  
then $T_*=-\infty$.
\end{thm}

\begin{proof}
The standard argument, contraction mapping principle by using Strichartz estimate, 
gives the time local well-posedness where the maximal existence 
time $T^*$ and $T_*$ depend on $\|f\|_{H^1}$ for the 
subcritical power $\alpha<\frac{4}{N-2}$ and on the profile of $f$ for the 
critical power  $\alpha=\frac{4}{N-2}$. 
In order to obtain the global soluvability, we deduce a priori estimate.
\begin{align*}
\frac12\frac{d}{dt}\|u(t)\|_{L^2}^2
&=\text{Re}\int u_t(t,x)\overline{u(t,x)}dx \\
&=\text{Re}\int(i\Delta u-i\lambda|u|^{\alpha}u)\overline{u}dx \\
&=\text{Re}\int i|\nabla u|^2-i\lambda|u|^{\alpha+2}dx  \\
&=\text{Im}\lambda\|u\|_{L^{\alpha+2}}^{\alpha+2}\ge0
\end{align*}
from the assumption Im$\lambda>0$. We also have
\begin{equation}\label{ap-est2}
\begin{aligned}
\frac12\frac{d}{dt}\|\nabla u(t)\|_{L^2}^2
&=-\text{Re}\int u_t\Delta\overline{u}dx \\
&=-\frac{\alpha+2}{2}\text{Re}(i\lambda)\int|u|^{\alpha}|\nabla u|^2
-\frac{\alpha}{2}\text{Re}(i\lambda)\int|u|^{\alpha-2}u^2(\nabla\bar{u})^2 \\
&\ge\frac{\alpha+2}{2}\text{Im}(\lambda)\int|u|^{\alpha}|\nabla u|^2
-\frac{\alpha}{2}|\lambda|\int|u|^{\alpha-2}u^2(\nabla\bar{u})^2 \\
&\ge\frac{(\alpha+2)\text{Im}(\lambda)-\alpha|\lambda|}{2}\int|u|^{\alpha}|\nabla u|^2dx
\ge0, 
\end{aligned}
\end{equation}
where the fourth line just used Re$[u^2(\nabla\bar{u})^2]\le|u|^2|\nabla\bar{u}|^2$. 
Therefore we have $\frac{d}{dt}\|u(t)\|_{H^1}\ge0$ and conclude 
global existence of $u$ on $(-\infty,0]$ for the subcritical power. 
With respect to the critical power, we may apply the argument in \cite{CMZ} 
with all $\lambda$ of \eqref{p-assume2} 
besides the critical complex coefficient $(\alpha+2)\text{Im}(\lambda)=\alpha|\lambda|$. 
Indeed we utilize \eqref{ap-est2} to have
\begin{align}
\int_{T_*}^0\|u(t)\|_{L^{\frac{N(\alpha+2)}{N-2}}}^{\alpha+2}dt
\le\frac{\|\nabla u(0)\|_{L^2}^2-\|\nabla u(T_*)\|_{L^2}^2}
{(\alpha+2)\text{Im}(\lambda)-\alpha|\lambda|}
<\infty
\end{align}
for the maximal existence time $T_*<0$. The index satisfies the embedding 
$W^{1,r}(\R^N)\embed L^{\frac{N(\alpha+2)}{N-2}}(\R^N)$ 
with Strichartz admissible $(r,\alpha+2)$ which implies $T_*=-\infty$. 
\end{proof}

\section{Estimates of $U$}
We estimate the profile $U$ in \eqref{b-profile}. This function satisfies the following ODE
\begin{align}
iU_t=\lambda|U|^{\alpha}U, \qquad t\not=0,x\in\R^N.
\end{align}
We collect the estimates on $U$ which are from \cite{CMZ} or slight modifications. 
We have for $1\le p\le\infty$ and any $t<0$,   
\begin{align}
\|U(t)\|_{L^p}&\lec(-t)^{-\frac{1}{\alpha}+\frac{N}{pk}}, \label{U-eq1} \\
\|\nabla U(t)\|_{L^p}&\lec(-t)^{-\frac{1}{\alpha}-\frac{1}{k}+\frac{N}{pk}}, \label{U-eq2}  \\
\|\Delta U(t)\|_{L^2}&\lec(-t)^{-\frac{1}{\alpha}-\frac{2}{k}+\frac{N}{2k}}, \label{U-eq3}  \\
\|\nabla\Delta U(t)\|_{L^2}&\lec(-t)^{-\frac{1}{\alpha}-\frac{3}{k}+\frac{N}{2k}},  \label{U-eq4} \\ 
\|(1+|\cdot|)\Delta U(t)\|_{L^2}
&\lec(-t)^{-\frac{1}{\alpha}-\frac{2}{k}+\frac{N}{2k}}. \label{U-eq5}
\end{align}
\begin{proof}
These estimates \eqref{U-eq1} -- \eqref{U-eq4} follows by the calculation in \cite{CMZ}. 
The estimate \eqref{U-eq5} is new and follow by the scaling argument as
\begin{align*}
\|(1+|\cdot|)\Delta U(t)\|_{L^2} \lec(1+(-t)^{\frac{1}{k}})\|\Delta U(t)\|_{L^2}
\end{align*}
and use \eqref{U-eq3} for small $t<0$. 
\end{proof}

\section{Difference between the solution and the profile}
Since the profile $U$ blows up at $t=0$, in order to obtain the blow up phenomena 
it suffice to estimate the difference between the solution of \eqref{nls} and 
the profile which converges to $0$ as $t\to0-0$. 
We actually discuss the approximate solutions with the initial data $U(T_n)$ at
the initial time $T_n=\frac1n$ for each $n=1,2,\ldots$. 
As we saw in section \ref{GWP} that \eqref{nls} is time globally wellposed, 
there is no need to worry about the degeneration of the existence times for the 
sequence of these solutions. 
We consider the Cauchy problem of \eqref{nls} with the initial data $U$ defined above 
at the initial time $T_n=-\frac1n$,
\begin{align}\label{C-nls}
  \left\{
    \begin{array}{l}
      u_t=i\Delta u-i\lambda|u|^{\alpha}u, \quad-\infty<t\le T_n, \ x\in\R^N, \\
      u(T_n)=U(T_n), \qquad x\in\R^N.
    \end{array}
  \right.
\end{align}
Theorem \ref{WP-thm} gives the unique existence of the global solution of \eqref{C-nls} 
for each $n$ 
\begin{align*}
u\in C((-\infty,T_n],H^1).
\end{align*}
We define and estimate the following
\begin{align*}
\e(t,x):=u(t,x)-U(t,x).
\end{align*}
Although it seems better to denote $u_n$ and $\e_n$ for $u$ and $\e$ respectively 
in each $n$, 
we abbreviate them when there is no confusion. Then $\e$ satisfies
\begin{align*}
\partial_t\e&=i\Delta\e-i\lambda|u|^{\alpha}u+i\lambda|U|^{\alpha}U+i\Delta U, 
\quad -\infty<t\le T_n, \ x\in\R^N, 
\\
\e(T_n)&=0, \qquad x\in\R^N. 
\end{align*}
\begin{prop}
There are positive constants $C_1, C_2, \mu, \gamma$ and $\delta$ 
such that for any $n$ 
\begin{align}\label{G1}
\|\e_n(t)\|_{H^1}&\le C_1(T_n-t)^{\mu}, \quad T_n-\delta\le t\le T_n, \\
\||x|\e_n(t)\|_{L^2}&\le C_2(T_n-t)^{\gamma}, \quad T_n-\delta\le t\le T_n. \label{G2}
\end{align}
\end{prop}

\begin{proof}
We start with the estimate on $L^2$ norm. 
\begin{align*}
\frac12\frac{d}{dt}\|\e(t)\|_{L^2}^2
&=\text{Re}\int(i\Delta\e-i\lambda(|u|^{\alpha}u-|U|^{\alpha}U)+i\Delta U)
\overline{\e}dx \\
&=\text{Re}\int(-i\lambda(|u|^{\alpha}u-|U|^{\alpha}U)+i\Delta U)
\overline{\e}dx=:I+\text{Re}\int i\Delta U\overline{\e}dx.
\end{align*}
We estimate $I$ that is the first term plus second term. 
We apply the mean value theorem for the two variables 
function $F(z)=F(z,\bar{z})=|z|^{\alpha}z$,
\begin{align*}
&|u|^{\alpha}u-|U|^{\alpha}U \\
&=\int_0^1F_z(U+\theta(u-U))(u-U)+F_{\bar{z}}(U+\theta(u-U))
(\overline{u-U})d\theta
\end{align*}
where $F_z=\frac{\alpha+2}{2}|z|^{\alpha}, F_{\bar{z}}=\frac{\alpha}{2}|z|^{\alpha-2}z^2$. 
We estimate
\begin{equation}\label{I2eq}
\begin{aligned}
-I&=\text{Re}\int\int_0^1 i\lambda
\frac{\alpha+2}2|U+\theta(u-U)|^{\alpha}|u-U|^2d\theta dx \\
&\qquad+\text{Re}\int\int_0^1 i\lambda
\frac{\alpha}2|U+\theta(u-U)|^{\alpha-2}(U+\theta(u-U))^2
(\overline{u-U})^2d\theta dx \\
&\le-\text{Im}(\lambda)\int\int_0^1
\frac{\alpha+2}2|U+\theta(u-U)|^{\alpha}|u-U|^2d\theta dx \\
&\qquad+|\lambda|\int\int_0^1\frac{\alpha}2|U+\theta(u-U)|^{\alpha}|u-U|^2d\theta dx
\le0
\end{aligned}
\end{equation}
where we used \eqref{p-assume1} at the last inequality. 
We therefore obtain 
\begin{align*}
\frac12\frac{d}{dt}\|\e(t)\|_{L^2}^2
\ge\text{Re}\int i\Delta U\overline{\e}dx
\ge-\|\Delta U\|_{L^2}\|\e\|_{L^2}.
\end{align*}
where we used Re$\int i\Delta\e\overline{\e}dx=0$ and $I_2\ge0$ in \eqref{I2eq}.
We may write
\begin{align*}
\frac{d}{dt}\|\e(t)\|_{L^2}\ge-\|\Delta U\|_{L^2}
\ge-C(-t)^{-\frac{1}{\alpha}-\frac{4-N}{2k}}. 
\end{align*}
We integrate this on the interval $(t,T_n)$ and apply \eqref{s-est2} to have
\begin{align}\label{eq01}
\|\e(t)\|_{L^2}\le C\left((-t)^{\mu_1}-(-T_n)^{\mu_1}\right)
\le C(T_n-t)^{\mu_1}
\end{align}
where
\begin{align*}
0<\mu_1=1-\frac{1}{\alpha}-\frac{4-N}{2k}<1
\end{align*}
for sufficiently large $k$. We next estimate $\dot{H}^1$ norm. 
\begin{align*}
\frac{d}{dt}\|\nabla\e(t)\|_{L^2}^2
&=\text{Re}\int\nabla(i\Delta\e-i\lambda(|u|^{\alpha}u-|U|^{\alpha}U)
+i\Delta U)\nabla\bar{\e}dx \\
&=\text{Re}\int(-i\lambda)\Big[\frac{\alpha+2}{2}(|u|^{\alpha}\nabla u-|U|^{\alpha}\nabla U) \\
&\qquad\qquad+\frac{\alpha}2(|u|^{\alpha-2}u^2\nabla\bar{u}-|U|^{\alpha-2}U^2\nabla\bar{U})
\Big]\nabla\overline{\e}
+i\nabla\Delta U\nabla\overline{\e}dx \\
&=\text{Re}\int(-i\lambda)\Big[\frac{\alpha+2}{2}(|u|^{\alpha}-|U|^{\alpha})\nabla U
+\frac{\alpha+2}{2}|u|^{\alpha}\nabla\e \\
&\qquad\qquad+\frac{\alpha}2(|u|^{\alpha-2}u^2-|U|^{\alpha-2}U^2)\nabla\bar{U}
+\frac{\alpha}{2}|u|^{\alpha-2}u^2\nabla\bar{\e}\Big]\nabla\overline{\e} \\
&\qquad\qquad+i\nabla\Delta U\nabla\overline{\e}dx.
\end{align*}
We estimate the sum of second and forth terms which is the worst if we consider 
the modulus of it  
in the sense of decay as $t\to0-0$. 
However it can be estimated since we just consider the real part of it likely as 
\begin{equation}\label{positive1}
\begin{aligned}
&\text{Re}\int(-i\lambda)\Big[\frac{\alpha+2}{2}|u|^{\alpha}|\nabla\e|^2
+\frac{\alpha}{2}|u|^{\alpha-2}u^2(\nabla\bar{\e})^2\Big]dx \\
&\ge\Big(\text{Im}(\lambda)\frac{\alpha+2}{2}-|\lambda|\frac{\alpha}{2}\Big)
\int|u|^{\alpha}|\nabla\e|^2dx. 
\end{aligned}
\end{equation}
Next 
we estimate the first term plus the third term where the coefficient $\lambda$ is 
estimated by its modulus.
We use Lemma \ref{s-est-Lem} for $\alpha\ge1$ to have
\begin{align*}
&||u|^{\alpha}-|U|^{\alpha}|+||u|^{\alpha-2}u^2-|U|^{\alpha-2}U^2|
\lec(|u|^{\alpha-1}+|U|^{\alpha-1})|u-U| \\
&\lec(|u-U|^{\alpha-1}+|U|^{\alpha-1})|u-U|=|\e|^{\alpha}+|U|^{\alpha-1}|\e|
\end{align*}
where we used $|u|^{\alpha-1}\lec|u-U|^{\alpha-1}+|U|^{\alpha-1}$ at the beginning of 
the second line. In the long run, we have
\begin{align}\label{eq03}
-\frac{d}{dt}\|\nabla\e(t)\|_{L^2}^2
\lec\int(|\e|^{\alpha}+|U|^{\alpha-1}|\e|)|\nabla U||\nabla\e|
+|\nabla\Delta U||\nabla\e|dx.
\end{align}
We estimate the first term in \eqref{eq03}. We separate it into the follwing two cases. 
\begin{align}\label{p-case1}
1<&\alpha\le\frac{N}{N-2}, \\
2\le &\alpha\le\frac{4}{N-2}. \label{p-case2}
\end{align} 
In the former case \eqref{p-case1}, we estimate
\begin{equation}\label{eq02}
\begin{aligned}
\int|\e|^{\alpha}|\nabla U||\nabla\e|dx
&\le\|\e\|_{L^r}^{\alpha}\|\nabla U\|_{L^\infty}\|\nabla\e\|_{L^2} \\
&\lec\|\e\|_{L^2}^{\alpha-\frac{N}2(\alpha-1)}\|\nabla U\|_{L^\infty}
\|\nabla\e\|_{L^2}^{1+\frac{N}2(\alpha-1)} \\
&\lec(-t)^{\mu_2}\|\nabla\e\|_{L^2}^{1+\frac{N}2(\alpha-1)}
\end{aligned}
\end{equation}
where we apply H\"older inequality and Gagliard-Nirenberg inequality with 
\begin{align*}
1=\frac{\alpha}{r}+\frac{1}{\infty}+\frac{1}{2}
\end{align*}
and
\begin{align*}
\frac1r=\frac12-\frac{\theta}{N}, \qquad\alpha\theta=\frac{N}2(\alpha-1)
\end{align*}
respectively. These satisfy
\begin{align*}
0<\theta=\frac{\alpha-1}{\alpha}\frac{N}{2}\le1
\end{align*}
which is provided by \eqref{p-case1}. 
We estimate the 
power of $t$ in \eqref{eq02}. From \eqref{U-eq2} and \eqref{eq01} we have
\begin{align*}
\mu_2=\Big(1-\frac{1}{\alpha}-\frac{4-N}{2k}\Big)
\Big(\alpha-\frac{N}2(\alpha-1)\Big)-\frac{1}{\alpha}-\frac{1}{k}.
\end{align*}
We require that this is strictly greater than $-1$. 
For simplisity we take $k=\infty$ and then $\mu_2>-1$ gives
\begin{align*}
(\alpha-1)\Big(\alpha+1-\frac{N}{2}(\alpha-1)\Big)>0
\end{align*}
and so
\begin{align}
1<\alpha<\frac{N+2}{N-2} \label{A2}
\end{align}
which is provided by \eqref{p-assume1}. 
In the latter case \eqref{p-case2}, we follow the arguments in \cite{CMZ}.
\begin{align*}
\int|\e|^{\alpha}|\nabla U||\nabla\e|dx
\lec\int|\e|(|U|^{\alpha-1}+|U+\e|^{\alpha-1})|\nabla U||\nabla\e|dx.
\end{align*}
The first term of this is the same with the second term in \eqref{eq03}, and we estimate later. 
We estimate the second term  of this by Cauchy Schwarz and with any $\delta>0$
\begin{align*}
&\int|\e||U+\e|^{\alpha-1}|\nabla U||\nabla\e|dx \\
&\le\left(\int|U+\e|^{\alpha}|\nabla\e|^2dx\right)^{\frac12}
\left(\int|\e|^2|U+\e|^{\alpha-2}|\nabla U|^2dx\right)^{\frac12} \\
&\le\delta\int|U+\e|^{\alpha}|\nabla\e|^2dx
+\frac{1}{4\delta}\int|\e|^2|U+\e|^{\alpha-2}|\nabla U|^2dx.
\end{align*}
We absorb the first term of this into \eqref{positive1}. 
We estimate the second term
\begin{align}\label{eq04}
\int|\e|^2|U+\e|^{\alpha-2}|\nabla U|^2dx
\lec\|\e\|_{L^2}^2\|U\|_{L^{\infty}}^{\alpha-2}\|\nabla U\|_{L^{\infty}}^2
+\|\e\|_{L^{\alpha}}^{\alpha}\|\nabla U\|_{L^{\infty}}^{2}
\end{align}
where we used $\alpha\ge2$. The Gagliardo-Nirenberg inequality also uses 
the condition $\alpha\ge2$
\begin{align*}
\|\e\|_{L^{\alpha}}^{\alpha}
\lec\|\e\|_{L^2}^{\frac{2N-\alpha(N-2)}{2}}\|\nabla\e\|_{L^2}^{\frac{N}{2}(\alpha-2)}.
\end{align*}
Therefore the right hand side of \eqref{eq04} is bounded by $C(-t)^{\mu_3}$ with 
some $C>0$ and $\mu_3>-1$. 
We also follow the estimates in \cite{CMZ} 
for the sencond and third terms in \eqref{eq03}. 
We use H\"older inequality and \eqref{U-eq1}, \eqref{U-eq2} and \eqref{U-eq4} 
to have
\begin{align*}
\int|U|^{\alpha-1}|\e||\nabla U||\nabla\e|dx
&\le\|U\|^{\alpha-1}_{L^{\infty}}\|\e\|_{L^2}\|\nabla U\|_{L^{\infty}}
\|\nabla\e\|_{L^2} \\
&\lec(-t)^{-\frac{\alpha-1}{\alpha}+\left(1-\frac{1}{\alpha}-\frac{4-N}{2k}\right)
-\frac{1}{\alpha}-\frac{1}{k}}\|\nabla\e\|_{L^2} \\
&=(-t)^{-\frac{1}{\alpha}-\frac{6-N}{2k}}\|\nabla\e\|_{L^2}
\end{align*}
and 
\begin{align*}
\int|\nabla\Delta U||\nabla\e|dx&\le\|\nabla\Delta U\|_{L^2}\|\nabla\e\|_{L^2} \\
&\lec(-t)^{-\frac{1}{\alpha}-\frac{6-N}{2k}}\|\nabla\e\|_{L^2}.
\end{align*}
We set $\mu_4=-\frac{1}{\alpha}-\frac{6-N}{2k}$ and $\mu_5=\min\{\mu_2, \mu_3, \mu_4\}$, 
we saw $\mu_5>-1$ for sufficiently large $k$. 
We then estimate \eqref{eq03} as
\begin{align*}
-\frac{d}{dt}\|\nabla\e(t)\|_{L^2}^2
\lec(-t)^{\mu_5}(1+\|\nabla\e\|_{L^2}^{1+\frac{N}{2}(\alpha-1)}). 
\end{align*}
This and $\e(T_n)=0$ give the value $\delta>0$ such that
\begin{align}\label{G3}
-\frac{d}{dt}\|\nabla\e(t)\|_{L^2}^2
\lec(-t)^{\mu_5}, \quad T_n-\delta<t<T_n 
\end{align}
uniformly with respect to $n$. Integrate this and we have \eqref{G1}. 
Next we show \eqref{G2}. We set $a>0$ and estimate
\begin{align*}
&\frac12\frac{d}{dt}\|e^{-a|\cdot|^2}|\cdot|\e\|_{L^2}^2
=\text{Re}\LR{\partial_t\e,e^{-2a|\cdot|^2}|\cdot|^2\e} \\
&=-\text{Re}\LR{i\nabla\e,(-4ax|x|^2e^{-2a|x|^2}+2xe^{-2a|x|^2})\e}
-\text{Re}\LR{ie^{-2a|x|^2}x\e,x\e} \\
&\qquad-\text{Re}\LR{i\lambda(|u|^{\alpha}u-|U|^{\alpha}U),e^{-2a|\cdot|^2}|\cdot|^2\e}
+\text{Re}\LR{i\Delta U,e^{-2a|\cdot|^2}|\cdot|^2\e}.
\end{align*}
The second term is zero and the third term is non negative from the same calculation 
with $I\ge0$ in \eqref{I2eq}. We only have to estimate the first term and fourth term 
like as
\begin{align*}
-\frac12\frac{d}{dt}\|e^{-a|\cdot|^2}|\cdot|\e\|_{L^2}^2
&\le2\|\nabla\e\|_{L^2}\|(1-2a|\cdot|^2)e^{-2a|\cdot|^2}|\cdot|\e\|_{L^2}
+\||\cdot|\Delta U\|_{L^2}\|e^{-2a|\cdot|}|\cdot|\e\|_{L^2} \\
&\le2\|\nabla\e\|_{L^2}\|e^{-a|\cdot|^2}|\cdot|\e\|_{L^2}
+\||\cdot|\Delta U\|_{L^2}\|e^{-a|\cdot|}|\cdot|\e\|_{L^2}.
\end{align*}
We then have from \eqref{G3} and \eqref{U-eq5}
\begin{align*}
-\frac{d}{dt}\|e^{-a|\cdot|^2}|\cdot|\e\|_{L^2}
&\lec\|\nabla\e\|_{L^2}+\||\cdot|\Delta U\|_{L^2} \\
&\lec(-t)^{-\frac{1}{\alpha}-\gamma}, \quad T_n-\delta<t<T_n.
\end{align*}
Integrate this to have
\begin{align*}
\|e^{-a|\cdot|^2}|\cdot|\e\|_{L^2}
\lec(T_n-t)^{1-\frac{1}{\alpha}-\gamma}, \quad T_n-\delta<t<T_n.
\end{align*}
We obtain the result \eqref{G2} by using Fatou's lemma as $a\to0+0$. 
\end{proof}

\section{Construction of blow up solution}

\begin{proof}
We construct the solution for Theorem \eqref{nls} from the 
approximate solution $u_n$ and the difference function $\e_n$ above. 
We set for $t>0$, 
\begin{align*}
v_n(t)=u_n(T_n-t), \qquad
V_n(t)=U(T_n-t), \qquad \eta_n(t)=\e_n(T_n-t).
\end{align*}
In the following, we consider $0<t<\delta$ since we had $T_n-\delta\le T_n-t\le T_n$ 
in \eqref{G1}. We also remember $T_n=-\frac{1}{n}<0$. 
Then we have $v_n=V_n+\eta_n$ and 
\begin{equation}\label{eta-eq1}
\begin{aligned}
\partial_t\eta_n&=-i\Delta\eta_n-|v_n|^{\alpha}v_n+|V_n|^{\alpha}V_n-i\Delta V_n \\
&=-i\Delta\eta_n-|V_n+\eta_n|^{\alpha}(V_n+\eta_n)+|V_n|^{\alpha}V_n-i\Delta V_n 
\end{aligned}
\end{equation}
on $0<t<\delta$. It holds by \eqref{G1} and \eqref{G2}
\begin{align}\label{H1bdd3}
\|\eta_n(t)\|_{H^1}+\||\cdot|\eta_n(t)\|_{L^2}\le Ct^{\mu}, \quad 0<t<\delta.
\end{align}
We estimate
\begin{align*}
\|V_n(t)\|_{L^p}\lec(t-T_n)^{-\frac{1}{\alpha}+\frac{N}{pk}}
=\Big(t+\frac{1}{n}\Big)^{-\frac{1}{\alpha}+\frac{N}{pk}}
\le t^{-\frac{1}{\alpha}+\frac{N}{pk}}
\end{align*}
for sufficiently large $k$ and any $n$. 
From the embeddings $H^1(\R^N)\embed L^{\alpha+2}(\R^N)$ and 
$L^{\frac{\alpha+2}{\alpha+1}}(\R^N)\embed H^{-1}(\R^N)$, 
\begin{align*}
&\||V_n+\eta_n|^{\alpha}(V_n+\eta_n)\|_{H^{-1}}\lec
\||V_n+\eta_n|^{\alpha}(V_n+\eta_n)\|_{L^{\frac{\alpha+2}{\alpha+1}}} \\
&\lec\|V_n\|_{L^{\alpha+2}}^{\alpha+1}+\|\eta_n\|_{L^{\alpha+2}}^{\alpha+1}
\lec\|V_n\|_{L^{\alpha+2}}^{\alpha+1}+\|\eta_n\|_{H^1}^{\alpha+1}
\lec t^{-\kappa},
\end{align*}
and therefore
\begin{align*}
\|\partial_t\eta_n\|_{H^{-1}}
&\lec\|\eta_n\|_{H^1}+\|V_n\|_{L^{\alpha+2}}^{\alpha+1}+\|\eta_n\|_{H^1}^{\alpha+1}
+\|\Delta V_n\|_{L^2} \\
&\lec t^{-\kappa}
\end{align*}
with some $\kappa>0$ for any $0<t\le\delta$ and any $n$. 
Given any $\tau\in(0,\delta)$, if we restrict the interval 
$(\tau,\delta)$, the sequence $\{\eta_n\}$ is bounded in 
$L^{\infty}(\tau,\delta:\Sigma)\cap W^{1,\infty}(\tau,\delta:H^{-1})$. 
Now we apply the Aubin-Lions lemma with
\begin{align*}
\Sigma\embed\embed L^2\embed H^{-1}
\end{align*} 
to conclude that 
there exists a subsequence which is still written by $\eta_n$ 
and the limits $\eta\in L^{\infty}(\tau,\delta,L^2)$ such that
\begin{align}\label{CL2limit2}
\|\eta_n-\eta\|_{L^{\infty}(\tau,\delta,L^2)}\to0.
\end{align}
We apply the diagonal argument. For sufficiently large $m$, we set 
$\sigma_k=\frac{1}{k}, k=m, m+1, m+2,\ldots$ such as $0<\sigma_k<\delta$. 
We obtain \eqref{CL2limit2} for each $\tau=\sigma_k$. We take the 
subsequence $\{\eta_{n_j}\}$ 
diagonally to have the limt $\eta$ which 
belongs to $L^{\infty}_{\rm loc}(0,\delta:L^2)$ and satisfies
\begin{align*}
\|\eta_{n_k}-\eta\|_{L^{\infty}(\tau,\delta,L^2)}\to0
\end{align*}
for any $0<\tau<\delta$. From this and the boundedness \eqref{H1bdd3}, 
we utilize Lemma \ref{con-abst} with $I=(\tau,\delta)$ to obtain 
four kinds of convergence \eqref{con1}, \eqref{con2}, \eqref{con3} and 
\eqref{con4} on the same $I=(\tau,\delta)$.
Therefore the limit $\eta$ satisfies the equation which corresponds to \eqref{eta-eq1},
\begin{align*}
\partial_t\eta=-i\Delta\eta-|U+\eta|^{\alpha}(U+\eta)+|V|^{\alpha}U-i\Delta U. 
\end{align*}
The function $u=U+\eta$ satisfies \eqref{nls} and for any $-\delta<t<0$,
\begin{align*}
\|u(t)-U(t)\|_{H^1}=\|\eta(-t)\|_{H^1}
\le\liminf_{k\to\infty}\|\eta_{n_k}(-t)\|_{H^1}\le C(-t)^{\mu}.
\end{align*}
\end{proof}


\end{document}